\documentclass[12pt,reqno]{amsart} 

\usepackage[utf8]{inputenc}
\usepackage[a4paper]{geometry}
\usepackage{amsmath}
\usepackage{amsfonts}
\usepackage{amssymb}
\usepackage{amsthm}

\usepackage{natbib}
\usepackage{color}
\usepackage[dvipsnames]{xcolor}
\usepackage[colorlinks=true,linkcolor=blue,citecolor=blue,pdfborder={0 0 0}]{hyperref}

\usepackage{dsfont}

\usepackage{graphicx}
\usepackage{enumerate}
\usepackage{paralist}

\definecolor{brown}{rgb}{0.59, 0.29, 0.0}

\theoremstyle{plain} 

\newtheorem{lem}{Lemma}[section] 
\newtheorem{prop}{Proposition}[section]

\theoremstyle{definition}
\newtheorem{defn}{Definition}[section]
\newtheorem{ex}{Example}[section]
\newcommand{\cond}{\textbf{C}\!\!}

\theoremstyle{remark}
\newtheorem{rem}{Remark}[section]

\newtheorem*{notation*}{Notation}

\numberwithin{equation}{section}

\newcommand{\nset}{\mathbb{N}}
\newcommand{\rset}{\mathbb{R}}

\newcommand{\qset}{\mathbb{Q}}
\newcommand{\ind}{\mathbf{1}}

\newcommand{\PP}{\Pr}
\newcommand{\PE}{\operatorname{\mathbb{E}}}

\newcommand{\ud}{\,\mathrm{d}} 
\newcommand{\point}{\,\cdot\,}

\newcommand{\domain}{\mathbb{D}}
\newcommand{\Espace}{\mathbb{E}}
\newcommand{\closeds}{\mathcal{F}}
\newcommand{\opens}{\mathcal{G}}
\newcommand{\compacts}{\mathcal{K}}

\newcommand{\Fdomain}{\domain\times\rset}
\newcommand{\uscD}{\mathrm{USC}(\domain)}
\newcommand{\borclosed}{\mathcal{B}(\closeds)}

\numberwithin{equation}{section}

\newcommand{\Fell}{\mathcal{F}}
\newcommand{\epi}{\operatorname{epi}}
\newcommand{\hypo}{\operatorname{hypo}}

\newcommand{\GEV}{\operatorname{GEV}}

\newcommand{\abs}[1]{\lvert{#1}\rvert}

\newcommand{\eps}{\varepsilon}
\newcommand{\eqd}{\stackrel{d}{=}}

\begin{document}
\title[Marginal standardization of upper semicontinuous
processes]{Marginal standardization of upper semicontinuous
  processes. With application to max-stable processes}

\author[A. Sabourin]{Anne Sabourin 
}
\author[J. Segers]{Johan Segers}
\address[A. Sabourin]{ LTCI, CNRS, Télécom ParisTech, Université Paris-Saclay, 75013,
  Paris, France}
\address[J. Segers]{Universit\'e catholique de Louvain, Institut de statistique, biostatistique et sciences actuarielles, Voie du Roman Pays 20, B-1348 Louvain-la-Neuve, Belgium.}

\date{\today}

\begin{abstract}
Extreme-value theory for random vectors and stochastic processes with continuous trajectories is usually formulated for random objects all of whose univariate marginal distributions are identical. In the spirit of Sklar's theorem from copula theory, such marginal standardization is carried out by the pointwise probability integral transform. Certain situations, however, call for stochastic models whose trajectories are not continuous but merely upper semicontinuous (usc). Unfortunately, the pointwise application of the probability integral transform to a usc process does in general not preserve the upper semicontinuity of the trajectories. In the present work, we give sufficient conditions for marginal standardization of usc processes to be possible, and we state a partial extension of Sklar's theorem for usc processes. We specialize the results to max-stable processes whose marginal distributions and normalizing sequences are allowed to vary with the coordinate.


\end{abstract}

\maketitle

\section{Introduction}


 

\bgroup
A common way to describe multivariate max-stable distributions is as follows: their margins are univariate max-stable distributions; after standardization of the marginal distributions to a common one, the joint distribution has a specific representation, describing the dependence structure. The separation into margins and dependence is in line with Sklar's theorem \cite{sklar59}, which provides a decomposition of a multivariate distribution into its margins and a copula, that is, a multivariate distribution with standard uniform margins. If the margins of the original distribution are continuous, the copula is unique and can be found by applying the probability integral transform to each variable. Conversely, to recover the original distribution, it suffices to apply the quantile transformation to each copula variable. Although it is more common in extreme-value theory to standardize to the Gumbel or the unit-Fr\'echet distribution rather than to the uniform distribution, the principle is the same. The advantage of breaking up a distribution into its margins and a copula is that both components can be modelled separately.

Applications in spatial statistics have spurred the development of extreme-value theory for stochastic processes. If the trajectories of the process are continuous almost surely, then the process can be reduced to a process with standardized margins and continuous trajectories via the probability integral transform applied to each individual variable. Conversely, the original process can be recovered from the standardized one by applying the quantile transform to each standardized variable. Moreover, these maps, sending one continuous function to another one, are measurable with respect to the sigma-field on the space of continuous functions that is generated by the finite-dimensional cylinders. These results hinge on the following two properties. First, the marginal distributions of a stochastic process with continuous trajectories depend continuously on the index variable. Second, the distribution of the random continuous function associated to the process is determined by the finite-dimensional distributions.

For stochastic processes with upper semicontinuous (usc) trajectories, however, the two properties mentioned above do not hold: the marginal distributions need not depend in a continuous way on the index variable, and the distribution of path functionals such as the supremum of the process is not determined by the finite-dimensional distributions of the process. Still, max-stable processes with usc trajectories have been proposed as models for spatial extremes of environmental variables \cite{davison2012geostatistics,huser2014space,schlather2002models}. As in the continuous case, construction of and inference on such models is carried out by a separation of concerns regarding the margins and the dependence structure. However, up to date, there is no theoretical foundation for such an approach. Another possible application of max-stable usc process is random utility maximization when the alternatives range over a compact metric space rather than a finite set \cite{mcfadden:1981,mcfadden:1989,resnick1991random}.

The present paper aims to fill the gap in theory and develop a framework for marginal standardization for stochastic processes with usc trajectories. For the mathematical framework, we follow \cite{norberg1987semicontinuous} and \cite{resnick1991random} and we work within the space $\uscD$ of usc functions on a locally compact subset $\domain$ of some Euclidean space. The space $\uscD$ is equipped with the hypo-topology: a usc function is identified with its hypograph, a closed subset of $\domain \times \rset$; the hypo-topology on $\uscD$ is then defined as the trace topology inherited from the Fell hit-and-miss topology on the space $\Fell$ of closed subsets of $\domain \times \rset$ \cite{salinetti1986convergence, vervaat:1986}.

The theory is specialized to max-stable usc processes. Our definition of max-stability allows the shape parameter of the marginal distributions to vary with the index variable of the stochastic process. As a consequence, the stabilizing affine transformations in the definition of max-stability may depend on the index variable too. However, the coordinatewise affine transformation of a usc function does not necessarily produce a usc function, so that care is needed in the formulation of the definition and the results. In \cite{dehaan:1984,resnick1991random}, in contrast, the marginal distributions are assumed to be Fr\'echet with unit shape parameter, so that the stabilizing sequences in the definition of max-stability are the same for all margins. The possibility to reduce the general case to  this simpler case is taken for granted, but, as argued in the paper, for usc processes, it is not.

Section~\ref{sec:uscProcesses} sets up the necessary background about usc processes. A general class of measurable transformations on the space of usc functions is introduced in Section~\ref{sec:transfo}. Under regularity conditions on the marginal distributions, this class includes the pointwise probability integral transform and its inverse. This property allows us to state a partial generalization of Sklar's theorem for usc processes in general (Section~\ref{sec:sklar}) and for max-stable ones in particular (Section~\ref{sec:max-stable}). Section~\ref{sec:conclusion} concludes. Some additional results are deferred to the appendices.
\egroup

\section{Random usc processes}
\label{sec:uscProcesses}

We review some essential definitions and properties of random usc
processes, or usc processes for short. The material in this section
may for instance be found in \cite{salinetti1986convergence},
\cite{beer1993topologies} (Chapter~5), \cite{vervaat:1988},~\cite{vervaat:holwerda:1997} and~\cite{molchanov2005theory}(Chapter~1.1 and Appendix~B).

Let $\domain$ be a non-empty, locally compact subset of some finite-dimensional Euclidean space. A function $x : \domain \to [-\infty, \infty]$ is \emph{upper semicontinuous} (usc) if and only $\limsup_{n \to \infty} x(s_n) \le x(s)$ whenever $s_n \to s \in \domain$. An equivalent property is that the set $\{ s \in \domain : \alpha \le x(s) \}$ is closed for each $\alpha \in \rset$. The definition of \emph{lower semicontinuous} (lsc) functions is similar, with a $\liminf$ rather than a $\limsup$ and with the inequalities being reversed.

For $x : \domain \to [-\infty, +\infty]$, the \emph{epigraph} and \emph{hypograph} are commonly defined by
\begin{align*} 
  \epi x 
  &= \{ (s, \alpha) \in \domain \times \rset : x(s) \le \alpha \}, \\
  \hypo x 
  &= \{ (s, \alpha) \in \domain \times \rset : \alpha \le x(s) \}. 
\end{align*}
Observe that $\epi x$ and $\hypo x$ are subsets of $\domain \times \rset$, even if the range of $x$ contains $-\infty$ or $+\infty$. A function is upper or lower semicontinuous if and only if its hypograph or epigraph, respectively, is closed.

Let $\uscD$ be the collection of all upper semicontinuous
functions from $\domain$ into $[-\infty, \infty]$. By identifying the
function $z \in \uscD$ with the set $\hypo z \subset \domain
\times \rset$, any topology on the space $\closeds = \closeds(\domain \times \rset)$ of closed subsets of $\domain \times \rset$ results in a trace topology on the space of usc functions.

The Fell \emph{hit-and-miss} topology on $\closeds$ is defined as follows. Let $\compacts$ and $\opens$ denote the families of compact and open subsets of $\domain\times\rset$, respectively.
A base for the Fell topology on $\closeds$ is the family of sets of the form
\[
  \closeds_{G_1,\dotsc,G_n}^K 
  = 
  \{ 
    F \in \closeds \; : \; 
    F \cap K = \varnothing, \,
    F \cap G_1 \neq \varnothing, \, \dotsc, \, F \cap G_n \neq \varnothing 
  \}
\]
for $K \in \compacts$ and $G_1, \ldots, G_n \in \opens$. A net of closed sets converges to a limit set $F$ if and only if every compact $K$ missed by $F$ is eventually also missed by the net and if every open $G$ hit by $F$ is eventually also hit by the net.

The Fell topology on $\closeds(\domain \times \rset)$ induces a trace
topology on $\uscD$, the \emph{hypo-topology}. 
Subsets of $\uscD$ which are open or closed with respect to the hypo-topology will be referred to as \emph{hypo-open} or \emph{hypo-closed}, respectively.
A sequence
$x_n$ in $\uscD$ is said to \emph{hypo-converge} to $x$ in
$\uscD$ if and only if $\hypo x_n$ converges to $\hypo x$ in
the Fell topology. Since the underlying space $\domain\times \rset$ is
locally compact, Hausdorff, and second countable (LCHS), the space
$\uscD$ thus becomes a compact, Hausdorff, second-countable space (see \emph{e.g.}\ \cite{molchanov2005theory}, Appendix, Theorem B.2). 
A convenient pointwise criterion for hypo-convergence is the following: a sequence $x_n \in \uscD$ hypo-converges to $x \in \uscD$ if, and only if,
\begin{equation}
\label{eq:hypoconv:pointwise}
  \left\{
  \begin{array}{l}
  \displaystyle
  \forall s \in \domain : \forall s_n \in \domain, \, s_n \to s : 
  \limsup_{n \to \infty} x_n(s_n) \le x(s); \\[1ex]
  \displaystyle
  \forall s \in \domain : \exists s_n \in \domain, \, s_n \to s : 
  \liminf_{n \to \infty} x_n(s_n) \ge x(s).
  \end{array}
  \right.
\end{equation}

Let $(\Omega,\mathcal{A}, \PP)$ be a complete probability space and
let $\borclosed$ be the Borel $\sigma$-field on $\closeds =
\closeds(\domain \times \rset)$ generated by the Fell topology. A
random closed set of $\domain \times \rset$ is a Borel measurable map
from $\Omega$ into $\closeds$. Similarly, a usc process $\xi$ on
$\domain$ is a Borel measurable map from $\Omega$ into
$\uscD$. Equivalently, $\hypo \xi$ is a random closed subset of
$\domain \times \rset$ taking values in the collection of all closed
hypographs. Some authors use the term `normal integrand' to refer to
such processes, as a special case of stochastic processes with usc
realizations (\cite{salinetti1986convergence}, p.\ 12). In this work, a usc process is always a function $\xi : \domain \times \Omega \to [-\infty,\infty]$ such that the map $\Omega \to \Fell : \omega \mapsto \hypo \xi(\point, \omega)$ is Borel measurable. 

The \emph{capacity functional} of $\hypo \xi$ is the map $T : \compacts \to [0, 1]$ defined by
\begin{equation*}
  T( K ) = \Pr\{ (\hypo \xi) \cap K \ne \varnothing \}, \qquad K \in \compacts.
\end{equation*}
We also call $T$ the capacity functional of $\xi$ rather than of $\hypo \xi$. Since $\Fdomain$ is LCHS, the Borel $\sigma$-field $\borclosed$ is generated by the sets 
\[
  \mathcal{F}_K 
  = 
  \{ F \in\closeds : F \cap K \neq\varnothing  \},
  \qquad K \in \compacts.
\]
(\cite{molchanov2005theory}, p. 2). 
The collection of complements of the sets $\mathcal{F}_K$ being a $\pi$-system, the capacity functional determines the distribution of a random closed set or a usc process.

We emphasize that the Borel $\sigma$-field on $\uscD$ is strictly larger than the one generated by the finite-dimensional cylinders. Without additional hypotheses, the finite-dimensional distributions of a usc process do not determine its distribution as a usc process. The evaluation mappings $\uscD \to [-\infty, \infty] : x \mapsto x(s)$ being hypo-measurable, a usc process $\xi$ is also a stochastic process (i.e., a collection of random variables) with usc trajectories. The converse is not true, however: for such stochastic processes, the map $\Omega \to \uscD : \omega \mapsto \xi(\,\cdot\,,\omega)$ is not necessarily hypo-measurable; see \cite{rockafellar1998variational}, Proposition~14.28, for a counter-example. Measurability of maps to or from $\uscD$ with respect to the Borel $\sigma$-field induced by the hypo-topology will be referred to as \emph{hypo-measurability}.

\section{Transformations of usc functions}
\label{sec:transfo}

Obviously, a usc function remains usc if multiplied by a positive, continuous function. More generally, we will need to prove that mappings derived from the probability integral transform or the quantile transform are hypo-measurable mappings from $\uscD$ to itself. To do so, we will rely on the function class $\mathcal{U}(\domain)$ defined and studied next.


\begin{defn}
\label{def:UD}
A function $U : \domain \times [-\infty, \infty] \to [-\infty, \infty]$ belongs to the class $\mathcal{U}(\domain)$ if it has the following two properties:
\begin{enumerate}[(a)]
\item
  For every $s \in \domain$, the map $x \mapsto U(s, x)$ is non-decreasing and right-continuous.
\item
  For every $x \in \rset \cup \{ \infty \}$, the map $s \mapsto U(s, x)$ is usc.
\end{enumerate}
\end{defn}

See Remark~\ref{rem:UD:usc} for the reason why we did not include $x = -\infty$ in condition~(b).

\begin{prop}
\label{prop:transfo}
Let $U \in \mathcal{U}(\domain)$.
\begin{enumerate}[(i)]
\item If $z_n$ hypo-converges to $z$ in $\uscD$ and if $s_n \to s$ in $\domain$ as $n \to \infty$, then $\limsup_{n \to \infty} U(s_n, z_n(s_n)) \le U(s, z(s))$.
\item For every $z \in \uscD$, the function $U_\ast(z) : \domain \to [-\infty, \infty]$ defined by $(U_\ast(z))(s) = U(s, z(s))$ belongs to $\uscD$ too.
\item For every compact $K \subset \domain \times \rset$, the set $\{ z \in \uscD : \hypo U_\ast(z) \cap K \neq \varnothing \}$ is hypo-closed.
\item The map $U_\ast : \uscD \to \uscD$ is hypo-measurable.
\end{enumerate}
\end{prop}

\begin{proof}
\textit{(i)}
Consider two cases: $z(s) = \infty$ and $z(s) < \infty$.

Suppose first that $z(s) = \infty$. By the assumptions on $U$, we have 
\[ 
  \limsup_{n \to \infty} U(s_n, z_n(s_n)) 
  \le 
  \limsup_{n \to \infty} U(s_n, \infty) 
  \le 
  U(s, \infty) 
  = 
  U(s, z(s)). 
\]

Suppose next that $z(s) < \infty$. Let $y > z(s)$. By hypo-convergence, we have $\limsup_{n \to \infty} z_n(s_n) \le z(s)$; see \eqref{eq:hypoconv:pointwise}. As a consequence, there exists a positive integer $n(y)$ such that $z_n(s_n) \le y$ for all $n \ge n(y)$. By the properties of $U$, we have
\[
  \limsup_{n \to \infty} U(s_n, z_n(s_n)) 
  \le 
  \limsup_{n \to \infty} U(s_n, y) 
  \le 
  U(s, y).
\] 
Since $y > z(s)$ was arbitrary and since $x \mapsto U(s, x)$ is non-decreasing and right-continuous, we find $\limsup_{n \to \infty} U(s_n, z_n(s_n)) \le \inf_{y > z(s)} U(s, y) = U(s, z(s))$.

\textit{(ii)}
In statement (i), set $z_n = z$ to see that $\limsup_{n \to \infty} U(s_n, z(s_n)) \le U(s, z(s))$ whenever $s_n \to s$ in $\domain$ as $n \to \infty$.

\textit{(iii)}
Let $K \subset \domain \times \rset$ be compact. Suppose that $z_n$ hypo-converges to $z$ in $\uscD$ as $n \to \infty$ and that $\hypo U_\ast(z_n) \cap K \neq \varnothing$ for all positive integers $n$. We need to prove that $\hypo U_\ast(z) \cap K \neq \varnothing$ too.

For each positive integer $n$, there exists $(s_n, x_n) \in K$ such that $U(s_n, z_n(s_n)) \ge x_n$. Since $K$ is compact, we can find a subsequence $n(k)$ such that $(s_{n(k)}, x_{n(k)}) \to (s, x) \in K$ as $k \to \infty$. By (i), we have 
\[ 
  U(s, z(s)) \ge \limsup_{k \to \infty} U(s_{n(k)}, z_{n(k)}(s_{n(k)})) \ge \limsup_{k \to \infty} x_{n(k)} = x. 
\]
As a consequence, $(s, x) \in \hypo U_\ast(z) \cap K$.

\textit{(iv)}
The collection of sets of the form $\{ z \in \uscD : \hypo(z) \cap K \ne \varnothing \}$, where $K$ ranges over the compact subsets of $\domain \times \rset$, generates the Borel $\sigma$-field on $\uscD$ induced by the hypo-topology. By (iii), the inverse image under $U_\ast$ of each such set is hypo-closed and thus hypo-measurable. We conclude that $U_\ast$ is hypo-measurable.
\end{proof}

\begin{rem}
\label{rem:UD:usc}
In part (b) of the definition of $\mathcal{U}(\domain)$, we did not include $x = -\infty$. The reason is that this case is automatically included: by Proposition~\ref{prop:transfo}(ii) applied to the function $z(s) \equiv -\infty$, the map $s \mapsto U(s, -\infty)$ is necessarily usc too.

Condition (b) is also necessary for the conclusion of Proposition~\ref{prop:transfo}(ii) to hold: given $x$, define $z(s) \equiv x$.
\end{rem}

A convenient property of $\mathcal{U}(\domain)$ is that it is closed under an appropriate kind of composition. This allows the deconstruction of complicated transformations into more elementary ones.

\begin{lem}\label{lem:compo}~
If~$ U$ and $V$ belong to $\mathcal{U}(\domain)$, the function $W$
defined by $(s, x) \mapsto V(s, U(s, x))$ belongs to
$\mathcal{U}(\domain)$ too, and $W_\ast = V_\ast \circ U_\ast$.
\end{lem}

\begin{proof}
First, fix $s \in \domain$. The map $x \mapsto V(s, U(s, x))$ is non-decreasing: for $-\infty \le x \le y \le \infty$, we have $U(s, x) \le U(s, y)$ and thus $V(s, U(s, x)) \le V(s, U(s, y))$. The map $x \mapsto V(s, U(s, x))$ is also right-continuous: if $x_n$ converges from the right to $x$ as $n \to \infty$, then so does $u_n = U(s, x_n)$ to $u = U(s, x)$ and thus $V(s, u_n)$ to $V(s, u)$.

Next, fix $x \in \rset \cup \{ \infty \}$. We need to show that the function $s \mapsto V(s, U(s, x))$ is usc. But the function $z$ defined by $s \mapsto U(s, x)$ is usc, and, by Proposition~\ref{prop:transfo}, so is the function $s \mapsto V(s, z(s)) = V(s, U(s, x))$.
\end{proof}

\begin{ex}
\label{ex:transfo:simple}
Here are some simple examples of functions $U$ in the class $\mathcal{U}(\domain)$ and the associated mappings $U_* : \uscD \to \uscD$.
\begin{enumerate}[(i)]
\item 
  If $f : [-\infty, \infty] \to [-\infty, \infty]$ is non-decreasing and right-continuous, the function $(s, x) \mapsto f(x)$ belongs to $\mathcal{U}(\domain)$. The associated map $\uscD \to \uscD$ is $z \mapsto f \circ z$.
\item
  If $y \in \uscD$, the functions $(s, x) \mapsto x \vee y(s)$ and $(s, x) \mapsto x \wedge y(s)$ both belong to $\mathcal{U}(\domain)$. The associated maps $\uscD \to \uscD$ are $z \mapsto z \vee y$ and $z \mapsto z \wedge y$, respectively.
\item
  If $a : \domain \to (0, \infty)$ is continuous, then the function $(s, x) \mapsto a(s) \, x$ belongs to $\mathcal{U}(\domain)$. If $b : \domain \to \rset$ is usc, then the map $(s, x) \mapsto x + b(s)$ belongs to $\mathcal{U}(\domain)$. The associated maps $\uscD \to \uscD$ are $z \mapsto az$ and $z \mapsto z + b$, respectively.
\end{enumerate}
\end{ex}

A more elaborate example of a function in $\mathcal{U}(\domain)$ is induced by the collection of right-continuous quantile functions of the marginal distributions of a stochastic process with usc trajectories. See Appendix~\ref{app:quantile} for some background on right-continuous quantile functions.

\begin{lem}
\label{lem:Qsp}
Let $\xi = ( \xi(s) : s \in \domain )$ be a stochastic process indexed
by $\domain$ and with values in $[-\infty, \infty]$. Let $F_s: x\in
[-\infty,\infty]\to F_s(x) = \PP[\xi(s)\le x]$ denote
the right-continuous marginal distribution function of $\xi(s)$. Define
\[
  Q_s(p) = \sup \{ y \in \rset : F_s(y)\le p \},
  \qquad (s, p) \in \domain \times [0, 1].
\]
If $\xi$ has usc trajectories, the function $(s, x) \mapsto Q_s( (x \vee 0) \wedge 1 )$ belongs to $\mathcal{U}(\domain)$.
\end{lem}

\begin{proof}
First, by Lemma~\ref{prop:quantile:right}, the map $[0, 1] \ni p \mapsto Q_s(p)$ is non-decreasing and right-continuous for every $s \in \domain$, and hence the same is true for the map $[-\infty, \infty] \ni x \mapsto Q_s( (x \vee 0) \wedge 1 )$.

Second, fix $x \in [-\infty, \infty]$ and write $p = (x \vee 0) \wedge 1 \in [0, 1]$. We need to show that the map $s \mapsto Q_s(p)$ is usc. Let $s_n \to s$ in $\domain$ as $n \to \infty$; we need to show that $Q_s(p) \ge \limsup_{n \to \infty} Q_{s_n}(p)$. Let $Q'$ be the right-continuous quantile function of the random variable $\limsup_{n \to \infty} \xi(s_n)$. The trajectories of $\xi$ are usc, and thus $\xi(s) \ge \limsup_{n \to \infty} \xi(s_n)$, which implies $Q_s(p) \ge Q'(p)$. By Lemma~\ref{lem:limsupX:limsupQp}, we then find $Q_s(p) \ge Q'(p) \ge \limsup_{n \to \infty} Q_{s_n}(p)$, as required.
\end{proof}

\section{Sklar's theorem for usc processes}
\label{sec:sklar}

A $d$-variate copula is the cumulative distribution function of a $d$-dimensional random vector with standard uniform margins. Sklar's \cite{sklar59} celebrated theorem states two things:
\begin{enumerate}[(I)]
\item
  For every copula $C$ and every vector $F_1, \ldots, F_d$ of univariate distribution functions, the function $(x_1, \ldots, x_d) \mapsto C(F_1(x_1), \ldots, F_d(x_d))$ is a $d$-variate distribution function with margins $F_1, \ldots, F_d$.
\item
  Every $d$-variate distribution function $F$ can be represented in this way.
\end{enumerate}
Reformulated in terms of random vectors, the two statements read as follows:
\begin{enumerate}[(I)]
\item
  For every random vector $(U_1, \ldots, U_d)$ with uniform components and for every vector $F_1, \ldots, F_d$ of univariate distribution functions, the random vector $(Q_1(U_1), \ldots, Q_d(U_d))$ has marginal distributions $F_1, \ldots, F_d$, where $Q_j$ is the (right- or left-continuous) quantile function corresponding to $F_j$.
\item
  Every random vector $(X_1, \ldots, X_d)$ can be represented in this way.
\end{enumerate}

We investigate up to what extent these statements hold for usc processes too. According to Proposition~\ref{prop:Sklar:I}, the first statement remains true provided the sections $s \mapsto Q_s(p)$ of the right-continuous quantile functions are usc. According to Proposition~\ref{prop:Sklar:II}, the Sklar representation is valid for usc processes whose marginal distribution functions have usc sections, and even then, the equality in distribution is not guaranteed. In Section~\ref{sec:max-stable}, we will specialize the two propositions to max-stable processes.

\begin{prop}[\`{a} la Sklar I for usc processes]
\label{prop:Sklar:I}
Let $Z$ be a usc process having standard uniform margins. Let $(F_s : s \in \domain)$ be a family of (right-continuous) distribution functions and let $Q_s(p) = \sup \{ x \in \rset : F_s(x) \le p \}$ for all $(s, p) \in \domain \times [0, 1]$. Define a stochastic process $\xi$ by $\xi(s) = Q_s((Z(s) \vee 0) \wedge 1)$ for $s \in \domain$. Then the following two statements are equivalent:
\begin{compactenum}[(i)]
\item
$\xi$ is a usc process with marginal distributions given by $F_s$.
\item
For every $p \in [0, 1]$, the function $s \mapsto Q_s(p)$ is usc.
\end{compactenum}
\end{prop}

\begin{proof}
For every $s \in \domain$, the random variable $(Z(s) \vee 0) \wedge 1$ is equal almost surely to $Z(s)$, so that its distribution is uniform on $[0, 1]$ too. By Lemma~\ref{prop:quantile:right}(iv), the distribution function of $\xi(s)$ is given by $F_s$. As a consequence, $Q_s$ is the right-continuous quantile function of the random variable $\xi(s)$.

If (i) holds, then the trajectories $s \mapsto \xi(s)$ are usc. By Lemma~\ref{lem:Qsp}, the function $s \mapsto Q_s(p)$ is usc for every $p \in [0, 1]$. 

Conversely, if (ii) holds, then the function $U$ defined by $(s, x) \mapsto Q_s((x \vee 0) \wedge 1)$ belongs to $\mathcal{U}(\domain)$ defined in Definition~\ref{def:UD}. Let $U_*$ be the associated map $\uscD \to \uscD$; see Proposition~\ref{prop:transfo}. Then $\xi = U_*(Z)$ is a usc process since $Z$ is a usc process and $U_*$ is hypo-measurable by Proposition~\ref{prop:transfo}(iv).
\end{proof}

\begin{prop}[\`{a} la Sklar II for usc processes] 
\label{prop:Sklar:II}
Let $\xi$ be a usc process. Let $F_s(x) = \Pr[ \xi(s) \le x ]$ for $x \in [-\infty, \infty]$ and let $Q_s(p) = \sup \{ x \in \rset : F_s(x) \le p \}$ for $p \in [0, 1]$. Suppose the following two conditions hold:
\begin{compactenum}[(a)] 
\item
For every $s \in \domain$, the distribution of $\xi(s)$ has no atoms in $[-\infty, \infty]$.
\item
For every $x \in \rset \cup \{+\infty\}$, the function $s \mapsto F_s(x)$ is usc.
\end{compactenum}
Then the following statements hold:
\begin{compactenum}[(i)]
\item
The process $Z$ defined by $Z(s) = F_s(\xi(s))$ is a usc process with standard uniform margins.
\item
The process $\tilde{\xi}$ defined by $\tilde{\xi}(s) = Q_s(Z(s)) = Q_s(F_s(\xi(s))$ is a usc process such that $\Pr[ \tilde{\xi}(s) = \xi(s) ] = 1$ for every $s \in \domain$. In particular, the finite-dimensional distributions of $\tilde{\xi}$ and $\xi$ are identical.
\end{compactenum}
\end{prop}

\begin{proof}
\textit{(i)}
By condition~(a), the marginal distributions of $Z$ are standard uniform. By condition~(b), the function $U$ defined by $(s, x) \mapsto F_s(x)$ belongs to $\mathcal{U}(\domain)$ (Definition~\ref{def:UD}) and we have $Z = U_*(\xi)$ with $U_* : \uscD \to \uscD$ as in Proposition~\ref{prop:transfo}. By item (iv) of that proposition, the map $U_*$ is hypo-measurable, so that $Z$ is a usc process too.

\textit{(ii)}
By Lemma~\ref{prop:quantile:right}(v), we have $\Pr[ \tilde{\xi}(s) = \xi(s) ] = 1$ for every $s \in \domain$. The function $Q_s$ is the right-continuous quantile function of $\xi(s)$ and the stochastic process $(\xi(s) : s \in \domain)$ has usc trajectories. Then requirement~(ii) of Proposition~\ref{prop:Sklar:I} is fulfilled by an application of Lemma~\ref{lem:Qsp}. By item (i) of the same proposition, $\tilde{\xi}$ is a usc process.
\end{proof}

\begin{rem}
Although $\Pr[\tilde{\xi}(s) = \xi(s)] = 1$ for all $s \in \domain$, it is not necessarily true that $\Pr[\tilde{\xi} = \xi] = 1$, and not even that $\tilde{\xi}$ and $\xi$ have the same distribution as usc processes: see Examples~\ref{ex:not_equal_in_law} and~\ref{ex:not_equal_in_law2}. However, if $\Pr[ \forall s \in \domain : 0 < F_s(\xi(s)) < 1 ] = 1$ and if $Q_s(F_s(x)) = x$ for every $s$ and every $x$ such that $0 < F_s(x) < 1$, then clearly $\Pr[ \tilde{\xi} = \xi ] = 1$.
\end{rem}

The following Lemma helps clarifying the meaning of the assumptions of
Proposition~\ref{prop:Sklar:II}. 
\begin{lem}[Regularity of the marginal distributions w.r.t.~the space variable ]\label{lem:equiv_usc_continuous_cdf}
If $\xi$ is a usc process, conditions (a) and (b) in Proposition~\ref{prop:Sklar:II} together imply that the map $s\mapsto F_s(x)=
\PP(\xi(s)\le x)$ 
is continuous,   for any fixed $x\in\rset$.  
\end{lem}

\begin{proof}
By condition~(a), we have $\PP[\xi(s)<x] = F_s(x)$ for each fixed $(s,x)$,  so that by Lemma~\ref{prop:quantile:right}(iii),  we have
\[ 
  \{ s \in \domain : x \le Q_s(p) \} 
  =
  \{ s \in \domain: F_s(x ) \le p \}
\]
for every $x \in [-\infty, \infty]$ and $p \in [0, 1]$. Since $\xi$ is a usc process, the function $s\mapsto Q_s(p)$ is usc for $p\in[0,1]$ (Lemma~\ref{lem:Qsp}). The set in the display is thus closed. But then the map $s \mapsto F_s(x)$ is \emph{lower} semicontinuous, and thus, by (b), continuous.
\end{proof}

\begin{rem}\label{rem:smearing}
Condition~(a) in Proposition~\ref{prop:Sklar:II}, continuity of the marginal distributions, can perhaps be avoided using an additional randomization device, `smearing out' the probability masses of any atoms, as in \cite{ruschendorf:2009}. However, even if this may ensure uniform margins, it may destroy upper semicontinuity, see Example~\ref{ex:atom}. Since our interest is in extreme-value theory, in particular in max-stable distributions (Section~\ref{sec:max-stable}), which are continuous, we do not pursue this issue further.
\end{rem}

\begin{rem}
Without condition~(b) in Proposition~\ref{prop:Sklar:II}, the trajectories of the stochastic process $(Z(s) : s \in \domain)$ may not be usc: see Example~\ref{ex:what_if_b_fails}. However, condition~(b) is not necessary either: see Example~\ref{ex:b_not_necess}.
\end{rem}

The examples below illustrate certain aspects of Propositions~\ref{prop:Sklar:I} and~\ref{prop:Sklar:II}. In the examples, we do not go into measurability issues, that is, we do not prove that the mappings $\xi$ from the underlying probability space into $\uscD$ are hypo-measurable. To do so, one can for instance rely on \cite{molchanov2005theory}(Example~1.2 on page~3 and Theorem~2.25 on page~37).

\begin{ex}[What may happen without (a) in Proposition~\ref{prop:Sklar:II}]
\label{ex:atom}
Consider two independent, uniformly distributed variables $X$ and $Y$ on $[0,1]$. Take $\domain = [-1,1]$, and define
\[
  \xi(s)
  =
  \begin{cases}
    \abs{s} X & \quad \text{ if $-1 \le s < 0$,} \\
    0 & \quad \text{ if $s = 0$,} \\
    s Y & \quad \text{ if $0 < s \le 1$.}
  \end{cases}
\]
Condition~(b) in Proposition~\ref{prop:Sklar:II} is fulfilled: the trajectories $s \mapsto F_s(x) = \Pr[ \xi(s) \le x ]$ are continuous for every $x \ne 0$ and still usc for $x = 0$. The marginal probability integral transform produces $Z(s) = X$ for $s < 0$ and $Z(s) = Y$ for $s > 0$, while $Z(0) = F_0(0) = 1$. The process $Z$ has usc trajectories but $Z(0)$ is not uniformly distributed. If one wanted to modify the definition of $Z$ at $0$ so that the modified version, $\tilde{Z}$, would still have usc trajectories and such that $\tilde{Z}(0)$ would be uniformly distributed, then one would need to have $\tilde{Z}(0) \ge X \vee Y$, contradicting the uniformity assumption. We conclude that it is impossible to write $\xi(s)$ as $Q_s( \tilde{Z}(s) )$, where $\tilde{Z}$ is a usc process with standard uniform margins.
\end{ex}

\begin{ex}[What may happen without (b) in Proposition~\ref{prop:Sklar:II}]
\label{ex:what_if_b_fails}
Consider two independent, uniformly distributed variables $X$ and $Y$ on $[0,1]$. On $\domain = [0,2]$, define 
\[
  \xi(s) 
  = 
  X \vee (Y \ind_{\{1\}}(s))
  =
  \begin{cases}
    X & \quad \text{ if $s \neq  1$,} \\
    X\vee Y& \quad \text{ if $s = 1$.}
  \end{cases}
\]
Then $\xi$ is a usc process and the distribution function, $F_s$, of $\xi(s)$ is given by
\begin{equation*}
  F_s(x) = \Pr[ \xi(s) \le x ] =
  \begin{cases}
    x & \quad \text{ if $s \neq 1$,} \\
    x^2& \quad \text{ if $s = 1$,}
  \end{cases}
\end{equation*}
where $x \in [0, 1]$. As a consequence, the function $s\mapsto F_s(x)$ is lsc but not usc if $0 < x < 1$, so that condition~(b) in Proposition~\ref{prop:Sklar:II} is violated. Reducing to standard uniform margins yields
\begin{equation*}
   Z(s) = F_s(\xi(s)) =
   \begin{cases}
     X & \quad \text{ if $s \neq 1$,} \\
     (X\vee Y)^2 & \quad \text{ if $s = 1$.} \\
   \end{cases}
 \end{equation*}
The event $\{0 < Y^2 < X < 1\}$ has positive probability, and on this event we have $(X\vee Y)^2 < X$, so that the trajectory $s \mapsto Z(s)$ is not usc. Hence, statement~(i) in Proposition~\ref{prop:Sklar:II} fails.
\end{ex}

Some additional examples are postponed to Appendix~\ref{ap:examplesSklar2}, which show that Condition~(b) in Proposition~\ref{prop:Sklar:II} is not necessary, and that the usc processes $\xi$ and $\tilde{\xi}$ in Proposition~\ref{prop:Sklar:II} may be different in law as random elements of $\uscD$.

\section{Max-stable processes}
\label{sec:max-stable}

We apply Propositions~\ref{prop:Sklar:I} and~\ref{prop:Sklar:II} to \emph{max-stable} usc processes. These processes and their standardized variants are introduced in Subsection~\ref{subsec:max-stable}, after some preliminaries on generalized extreme-value distributions in Subsection~\ref{subsec:gev}. The main results are given in Subsection~\ref{subsec:max-stable:sklar}, followed by some examples in Subsection~\ref{subsec:max-stable:examples}.

\subsection{Generalized extreme-value distributions}
\label{subsec:gev}

The distribution function of the generalized extreme-value (GEV) distribution with parameter vector $\theta = (\gamma, \mu, \sigma) \in \Theta = \rset \times \rset \times (0, \infty)$ is given by
\begin{equation}
  \label{eq:GEV-distrib}
  F(x; \theta) =
  \begin{cases}
    \exp [ - \{ 1 + \gamma (x-\mu)/\sigma \}^{-1/\gamma} ] & \text{if $\gamma \neq 0$ and $\sigma + \gamma(x - \mu) > 0$,} \\
    \exp [ - \exp \{ - (x-\mu)/\sigma \} ] & \text{if $\gamma = 0$ and $x \in \rset$.}
  \end{cases}  
\end{equation}
The corresponding quantile function is
\begin{equation}
\label{eq:gev:Q}
  Q(p; \theta) =
  \begin{cases}
    \mu + \sigma [\{-1/\log(p)\}^\gamma - 1] / \gamma & \text{if $\gamma \neq 0$,} \\
    \mu + \sigma \log \{ - 1/\log(p) \} & \text{if $\gamma = 0$,}
  \end{cases}
\end{equation}
for $0 < p < 1$. The support is equal to the interval $\{ x \in \rset : \sigma + \gamma(x - \mu) > 0 \}$. In particular, the lower endpoint is equal to $Q(0; \theta) = -\infty$ if $\gamma \le 0$ and $Q(0; \theta) = \mu - \sigma/\gamma$ if $\gamma > 0$.

For every $n$, there exist unique scalars $a_{n,\theta} \in (0, \infty)$ and $b_{n,\theta} \in \rset$ such that the following max-stability relation holds:
\begin{equation}
\label{eq:F_max-stable}
  F^n(a_{n,\theta} x + b_{n,\theta}; \theta) = F(x; \theta), \qquad x \in \rset.
\end{equation}
In fact, a non-degenerate distribution is max-stable if and only if it is GEV. This property motivates the use of such distributions for modeling maxima over many variables. The location and scale sequences are given by
\begin{align}
\label{eq:anbn}
  a_{n,\theta} &= n^\gamma, &
  b_{n,\theta} &= 
  \begin{cases}
  (\sigma - \gamma \mu) (n^\gamma - 1)/\gamma &\text{if $\gamma \ne 0$,} \\
  \sigma \, \log n &\text{if $\gamma = 0$.}
  \end{cases}
\end{align}
For quantile functions, max-stability means that
\begin{equation}
\label{eq:Q:max_stable}
  Q(p^{1/n}; \theta) 
  = a_{n,\theta} \, Q(p; \theta) + b_{n,\theta}, 
  \qquad p \in [0, 1].
\end{equation}
Note that $Q(e^{-1}; \theta) = \mu$ and thus $Q(e^{-1/n}; \theta) = a_n \, \mu + b_n$.

In Sklar's theorem, the uniform distribution on $[0, 1]$ plays the role of pivot. Here, it is natural to standardize to a member of the GEV family. Multiple choices are possible. We opt for the unit-Fr\'{e}chet distribution, given by
\begin{equation}
\label{eq:unitFrechet}
  \Phi(x) = F(x; 1, 1, 1) =
  \begin{cases}
    0 & \text{if $x \le 0$}, \\
    e^{-1/x} & \text{if $x > 0$.}
  \end{cases}
\end{equation}
The unit-Fr\'{e}chet quantile function is given by $Q(p; 1, 1, 1) = -1/\log(p)$ for $0 < p < 1$. If the law of $X$ is unit-Fr\'{e}chet, then the law of $Q(\Phi(X); \theta)$ is $\GEV(\theta)$.

Consider now a stochastic process $\{\xi(s), \, s\in\domain\}$, such that $\xi(s)$ follows a GEV distribution with parameter $\theta$ depending on $s$, so that $\PP[\xi(s)\le x] = F(x; \theta(s))$, with $F$ as in~\eqref{eq:GEV-distrib}. In view of the conditions in Propositions~\ref{prop:Sklar:I} and~\ref{prop:Sklar:II}, the following lemma is relevant.

\begin{lem}
\label{lem:theta_cnt}
Let $\theta : \domain \to \Theta$. The functions $s \mapsto F(x; \theta(s))$ and $s \mapsto Q(p; \theta(s))$ are usc for every $x \in \rset$ and $p \in [0, 1]$ if and only if $\theta$ is continuous.
\end{lem}

\begin{proof}
If $\theta$ is continuous, then the maps $s \mapsto F(x; \theta(s))$ and $s \mapsto Q(p; \theta(s))$ are continuous and thus usc.

Conversely, suppose that the functions $s \mapsto F(x; \theta(s))$ and
$s \mapsto Q(p; \theta(s))$ are usc for every $x \in
\rset$ and $p \in [0, 1]$. The argument is similar to the proof of
Lemma~\ref{lem:equiv_usc_continuous_cdf}. By Lemma~\ref{prop:quantile:right}(iii), we have 
\[ 
  \{ s \in \domain : x \le Q(p; \theta(s)) \} 
  =
  \{ s \in \domain: F(x; \theta(s)) \le p \}
\]
for every $x \in [-\infty, \infty]$ and $p \in [0, 1]$; note that GEV distribution functions are continuous functions of $x$. Since the map $s \mapsto Q(p; \theta(s))$ is usc, the above sets are closed. But then the map $s \mapsto F(x; \theta(s))$ must be lsc, and thus continuous.
Finally, the map sending a GEV distribution to its parameter is continuous with respect to the weak topology (see Lemma~\ref{lem:continuityGEVparameter}). Hence, $\theta$ is continuous.
\end{proof}

\begin{rem}
\label{rem:theta_cnt}
If $\theta : \domain \to \Theta$ is continuous, then the normalizing functions $s \mapsto a_{n,\theta(s)}$ and $s \mapsto b_{n,\theta(s)}$ are continuous as well. Consider the map $U_n(s, x) = \{ x - b_{n,\theta(s)} \} / a_{n,\theta(s)}$, for $(s, x) \in \domain \times [-\infty, \infty]$. Clearly, $U_n \in \mathcal{U}(\domain)$. By Proposition~\ref{prop:transfo}, the transformation $\uscD \to \uscD : z \mapsto (z - b_{n,\theta}) / a_{n,\theta}$ is then well-defined and hypo-measurable.
\end{rem}

\subsection{Max-stable usc processes}
\label{subsec:max-stable}

Max-stable usc processes have been defined in the literature \cite{molchanov2005theory,vervaat:1986,vervaat1988random} as usc processes whose distribution is invariant under the componentwise maximum operation, up to an affine rescaling involving \emph{constants} which are not allowed to depend upon the index variable $s \in \domain$. Likewise, the processes in \cite{dehaan:1984} and \cite{resnick1991random} and have marginal distributions which are Fr\'{e}chet with shape parameter $\gamma \equiv 1$ and lower endpoint $\mu - \gamma/\sigma \equiv 0$, so that the scaling sequences are $a_n \equiv n$ and $b_n \equiv 0$. Moreover, max-stability is defined in \cite{dehaan:1984} in terms of finite-dimensional distributions only. Definition~\ref{def:max-stable} below extends these previous approaches by allowing for an index-dependent rescaling, through scaling \emph{functions} $a_n$ and $b_n$, and by viewing the random objects as random elements in $\uscD$.

\begin{defn}
\label{def:max-stable}
A usc process $\xi$ is \emph{max-stable} if, for all integer $n\ge 1$, there exist functions $a_n : \domain \to (0, \infty)$ and $b_n : \domain \to \rset$ such that, for each vector of $n$ independent and identically distributed (iid) usc processes $\xi_1,\ldots,\xi_n$ with the same law as $\xi$, we have
\begin{equation}
  \label{eq:xi_max-stable}
  \textstyle \bigvee_{i=1}^n \xi_i \eqd a_n \xi + b_n \quad\text{ in $\uscD$}. 
\end{equation}
A max-stable usc process $\xi^*$ is said to be \emph{simple} if, in addition, its marginal distributions are unit-Fr\'{e}chet as in \eqref{eq:unitFrechet}. 
In that case, the norming functions are given by $a_n(s) = n$ and $b_n(s) = 0$ for all $n \ge 1$ and $s \in \domain$, i.e., for iid usc processes $\xi_1^*, \ldots, \xi_n^*$ with the same law as $\xi^*$, we have
\begin{equation}
  \label{eq:defSimpleMStable}
  \textstyle \bigvee_{i=1}^n \xi^*_i \eqd n \xi^* \quad\text{ in $\uscD$}.
\end{equation}
\end{defn}

In \eqref{eq:xi_max-stable} and \eqref{eq:defSimpleMStable}, the meaning is that the induced probability distributions on the space $\uscD$ equipped with the sigma-field of hypo-measurable sets are equal.
In Definition~\ref{def:max-stable}, it is implicitly understood that the functions $a_n$ and $b_n$ are such that the right-hand side of \eqref{eq:xi_max-stable} still defines a usc process. If $a_n$ is continuous and $b_n$ is usc, then this is automatically the case; see Lemma~\ref{lem:compo} and Example~\ref{ex:transfo:simple}(iii).

Equation~\eqref{eq:xi_max-stable} is not necessarily the same as saying that $( \bigvee_{i=1}^n \xi_i - b_{n} )  / a_{n}$ is equal in distribution to $\xi$. The reason is that it is not clear that $(\bigvee_{i=1}^n \xi_i - b_{n})/a_{n}$ is a usc process. Whether this is the case or not remains an open question.

The evaluation map $\uscD \to [-\infty, \infty] : z \mapsto z(s)$ is hypo-measurable for all $s \in \domain$. Equation~\eqref{eq:xi_max-stable} then implies the following distributional equality between random variables:
\[
  \textstyle
  \bigvee_{i=1}^n \xi_i(s) \eqd a_n(s) \, \xi(s) + b_n(s), \qquad s \in \domain. 
\]
As a consequence, the marginal distribution of $\xi(s)$ is max-stable and therefore GEV with some parameter vector $\theta(s) \in \Theta$. The normalizing functions $a_n$ and $b_n$ must then be of the form $a_n(s) = a_{n,\theta(s)}$ and $b_n(s) = b_{n,\theta(s)}$ as in~\eqref{eq:anbn}. If $\theta : \domain \to \Theta$ is continuous, then the normalizing functions are continuous too, and, by Remark~\ref{rem:theta_cnt}, the usc process $(\bigvee_{i=1}^n \xi_i - b_n) / a_n$ is well-defined and equal in distribution to $\xi$.

\subsection{Sklar's theorem for max-stable usc processes}
\label{subsec:max-stable:sklar}

We investigate the relation between general and simple max-stable usc processes via the pointwise probability integral transform and its inverse. Max-stability of usc processes is defined in \eqref{eq:xi_max-stable} via an equality of distributions on $\uscD$ rather than of finite-dimensional distributions. It is therefore not clear from the outset that max-stability is preserved by pointwise transformations.

Proposition~\ref{prop:Sklar:I:max-stable} gives a necessary and sufficient condition on the GEV margins to be able to construct a general max-stable usc process starting from a simple one. Propositions~\ref{prop:Sklar:II:GEV} and~\ref{prop:Sklar:II:max-stable} treat the converse question, that is, when can a max-stable usc process be first reduced to a simple one and then be reconstructed from it.

\begin{prop}[\`{a} la Sklar~I for max-stable usc processes]
\label{prop:Sklar:I:max-stable}
Let $\xi^*$ be a simple max-stable usc process. Let $\theta : \domain \to \Theta$. Define a stochastic process $\xi$ by $\xi(s) = Q( \Phi(\xi^*(s)); \theta(s) )$ for $s \in \domain$. Then the following two statements are equivalent:
\begin{enumerate}[(i)]
\item
  $\xi$ is a usc process with marginal distributions $\GEV(\theta(s))$.
\item
  For every $p \in [0, 1]$, the function $s \mapsto Q(p; \theta(s))$ is usc.
\end{enumerate}
If these conditions hold, then $\xi$ is a max-stable usc process with normalizing functions $a_n(s) = a_{n,\theta(s)}$ and $b_n(s) =
b_{n,\theta(s)}$.
\end{prop}

\begin{proof}
By Proposition~\ref{prop:Sklar:II}(i) applied to $\xi^*$, the stochastic process $Z(s) = \Phi(\xi^*(s))$ induces a usc process whose margins are uniform on $[0, 1]$. The equivalence of statements (i) and (ii) then follows from Proposition~\ref{prop:Sklar:I}.

Assume that (i) and (ii) are fulfilled. We need to show that the right-hand side in \eqref{eq:xi_max-stable} defines a usc process and that the stated equality in distribution holds.

For positive integer $n$, define $U_n(s, x) = Q(\Phi(nx); \theta(s))$ for $(s, x) \in \domain \times [-\infty, \infty]$. In view of Lemma~\ref{lem:compo} and Example~\ref{ex:transfo:simple}, the map $U_n$ belongs to $\mathcal{U}(\domain)$. Moreover, max-stability \eqref{eq:Q:max_stable} implies
\[
  U_n(s, x) = Q( \Phi(x)^{1/n}; \theta(s) ) 
  = a_{n,\theta(s)} \, Q( \Phi(x); \theta(s) ) + b_{n,\theta(s)}.
\]
It follows that
\[
  a_{n,\theta(s)} \, \xi(s) + b_{n,\theta(s)}
  =
  U_n(s, \xi^*(s)).
\]
By Proposition~\ref{prop:transfo}, the function $s \mapsto U_n(s,
z(s))$ belongs to $\uscD$ for every $z \in \uscD$, and the map
$U_{n,*}$ from $\uscD$ to itself sending $z \in \uscD$ to this
function is hypo-measurable. We conclude that $a_n \, \xi + b_n$ is a usc process.

Next, we prove that $\xi$ is max-stable. Let $\xi_1,\ldots,\xi_n$ be iid usc processes with the same law as $\xi$. Further, let $\xi^*_1,\ldots,\xi^*_n$ be iid usc processes with the same law as $\xi^*$. For every $i \in \{1, \ldots, n\}$, we have
\[
  \xi_i \eqd \xi = U_{1,*}(\xi^*) \eqd  U_{1,*}(\xi^*_i).
\]
The last equality in distribution comes from the hypo-measurability of $U_{1,*}$. By independence, it follows that
\[
  (\xi_1,\ldots,\xi_n) 
  \eqd
  \left(U_{1,*}(\xi^*_1),\ldots,U_{1,*}(\xi^*_n) \right).
\]
Write $(\tilde\xi_1, \ldots, \tilde\xi_n) = (U_{1,*}(\xi^*_i),\ldots, U_{1,*}(\xi^*_n))$. By monotonicity, we have, for $s\in\domain$,
\[
  \textstyle
  \bigvee_{i=1}^n \tilde\xi_i(s) 
  = 
  \bigvee_{i=1}^n Q \bigl( \Phi(\xi^*_i(s)); \theta(s) \bigr)
  =
  Q \bigl( \Phi (\bigvee_{i=1}^n \xi^*_i(s)); \theta(s) \bigr).
\]
Since $\xi^*$ is a simple max-stable usc process, we have $\bigvee_{i=1}^n \xi^*_i \eqd n\xi^*$ in $\uscD$. But then also
\begin{align*}
  \textstyle\bigvee_{i=1}^n \xi_i 
  &\eqd \textstyle \bigvee_{i=1}^n \tilde\xi_i\\
  &\eqd
  Q \bigl( \Phi( n \xi^* ); \theta \bigr)
  =
  Q \bigl( \Phi( \xi^* )^{1/n}; \theta \bigr) \\
  &= 
  a_{n,\theta} \, Q \bigl( \Phi( \xi^* ); \theta \bigr) + b_{n,\theta}
  =
  a_{n,\theta} \, \xi + b_{n,\theta}.
\end{align*}
We conclude that $\xi$ is max-stable, as required.
\end{proof}

\begin{prop}[\`{a} la Sklar II for usc processes with GEV margins]
\label{prop:Sklar:II:GEV}
Let $\xi$ be a usc process with $\GEV(\theta(s))$ margins for $s \in \domain$. If $\theta : \domain \to \Theta$ is continuous, then the following statements hold:
\begin{enumerate}[(i)]
\item 
The process $\xi^*$ defined by $\xi^*(s) = -1/\log F(\xi(s); \theta(s))$ is a usc process with unit-Fr\'{e}chet margins.
\item
The process $\tilde{\xi}$ defined by $\tilde{\xi}(s) = Q(\Phi(\xi^*(s)); \theta(s))$ is a usc process and, with probability one,
\[
  \forall s \in \domain:
  \tilde{\xi}(s)
  =
  \begin{cases}
    \xi(s) & \text{if $\xi^*(s) < \infty$,} \\
    \infty & \text{if $\xi^*(s) = \infty$.}
  \end{cases}
\]
\end{enumerate}
\end{prop}

\begin{proof}
The marginal distributions of $\xi$ are GEV and depend continuously on $s$. Conditions~(a) and~(b) in Proposition~\ref{prop:Sklar:II} are therefore satisfied.

By Proposition~\ref{prop:Sklar:II}(i), the process $Z$ defined by $Z(s) = F(\xi(s); \theta(s))$ is a usc process with standard uniform margins. It then follows that $\xi^*$ defined by $\xi^*(s) = - 1/\log Z(s)$ is a usc process too, its margins being unit-Fr\'{e}chet.

Since $0 \le Z(s) \le 1$ by construction, we have $0 \le \xi^*(s) \le \infty$ and thus $\Phi(\xi^*(s)) = Z(s)$. We find $\tilde{\xi}(s) = Q(Z(s); \theta(s)) = Q(F(\xi(s); \theta(s)); \theta(s))$. By Proposition~\ref{prop:Sklar:II}(ii), the process $\tilde{\xi}$ is a usc process.

Recall that GEV distribution functions are continuous and strictly increasing on their domains. As a consequence, for all $x$ such that $x \ge Q(0; \theta(s))$, we have 
\[
  Q(F(x; \theta(s)); \theta(s)) 
  = 
  \begin{cases}
    x & \text{if $F(x; \theta(s)) < 1$,} \\
    \infty & \text{if $F(x; \theta(s)) = 1$.}
  \end{cases}
\]
Moreover, Lemma~\ref{lem:lower_bound} implies that $\xi \ge Q(0; \theta)$ almost surely. Since $\xi^*(s) < \infty$ if and only if $F(\xi(s); \theta(s)) < 1$, we arrive at the stated formula for $\tilde{\xi}$.
\end{proof}

\begin{prop}[\`{a} la Sklar II for max-stable processes]
\label{prop:Sklar:II:max-stable}
Let $\xi$ be a usc process with $\GEV(\theta(s))$ margins for $s \in \domain$. 
Assume that for every compact $K \subset \domain$, we have $\sup_{s \in K} F(\xi(s); \theta(s)) < 1$ with probability one.
As in Proposition~\ref{prop:Sklar:II:GEV}, define two usc processes $\xi^*$ and $\tilde \xi$ by $\xi^*(s) = -1/\log F(\xi(s); \theta(s))$ and $\tilde \xi(s) = Q( \Phi(\xi^*(s)); \theta(s))$, for $s\in\domain$. Then, almost surely, $\xi = \tilde\xi$. Furthermore, the following two statements are equivalent:
\begin{compactenum}[(i)]
\item The usc process $\xi$ is max-stable. 
\item The usc process $\xi^*$ is simple max-stable. 
\end{compactenum}
\end{prop}

\begin{proof}
The fact that $\xi^*$ is a usc process with unit-Fr\'{e}chet margins is a consequence of the continuity of $\theta$ and Proposition~\ref{prop:Sklar:II:GEV}(i). Let $(K_n)_{n \in \nset}$ be a compact cover of $\domain$; the existence of a compact cover is guaranteed by the fact that $\domain$ is locally compact and second countable, whence Lindel\"{o}f. We have $\Pr[ \forall n \in \nset : \sup_{s \in K_n} F(\xi(s); \theta(s)) < 1 ] = 1$ by assumption and thus also $\Pr[ \forall s \in \domain : F(\xi(s); \theta(s)) < 1 ] = 1$. Proposition~\ref{prop:Sklar:II:GEV}(ii) implies that $\xi = \tilde{\xi}$ almost surely.

The functions $s \mapsto a_{n,\theta(s)}$ and $s \mapsto b_{n,\theta(s)}$ are continuous. The map $\domain \times [-\infty, \infty] \to [-\infty, \infty]$ defined by $(s, x) \mapsto a_{n,\theta(s)} \, x + b_{n,\theta(s)}$ belongs to $\mathcal{U}(\domain)$. By Proposition~\ref{prop:transfo}, the map from $\uscD$ to itself sending $z$ to the function $s \mapsto a_{n,\theta(s)} \, z(s) + b_{n,\theta(s)}$ is well-defined and hypo-measurable. It follows that $a_{n,\theta} \, \xi + b_{n,\theta}$ is a usc process.

Suppose first that~(ii) holds.  By Proposition~\ref{prop:Sklar:I:max-stable}, the usc process $\tilde\xi$ is max-stable with norming functions $s \mapsto a_{n,\theta(s)}$ and $s \mapsto b_{n,\theta(s)}$. Since $\xi$ and $\tilde\xi$ are equal almost surely in $\uscD$, they are also in equal in law. Statement~(i) follows.

Conversely, suppose that (i) holds. 
Let $(\xi^*_i )_{i=1}^n$ and $(\xi_i)_{i=1}^n$ be vectors of iid usc processes with common laws equal to the ones of $\xi^*$ and $\xi$, respectively. For $z\in\uscD$, write $ -1/\log F(z,\theta) = (-1/\log F(z(s),\theta(s) )_{s\in\domain}$. The mapping $z\mapsto -1/\log F(z,\theta)$ from $\uscD$ to itself is hypo-measurable, by an argument as in the proof of Proposition~\ref{prop:Sklar:II:GEV}.  For $i\in\{1,\ldots,n\}$, we have, in $\uscD$,
\[
  \xi^*_i \eqd \xi^* = -1/\log F(\xi,\theta) \eqd -1/\log F(\xi_i,\theta).
\]
By independence, we  thus have $(\xi^*_i )_{i=1}^n \eqd (-1/\log F(\xi_i,\theta))_{i=1}^n $. Property~(i) and max-stability \eqref{eq:F_max-stable} now say that
\begin{align*}
  \textstyle 
  \bigvee_{i=1}^n \xi^*_i 
  & \eqd 
  \textstyle 
  -1 / \log F( \bigvee_{i=1}^n \xi_i; \theta ) \\
  &\eqd
  \textstyle 
  -1 / \log F( a_{n,\theta} \xi + b_{n,\theta} ; \theta ) \\
  &=
  \textstyle 
  -1 / \log \{ F( \xi ; \theta )^{1/n} \} 
  =
  \textstyle 
  n \, [-1/\log F( \xi; \theta)] 
  =
  n \xi^*. 
\end{align*}
We conclude that $\xi^*$ is simple max-stable.
\end{proof}

\begin{rem}[Regarding the finiteness of $\xi^*$ in Proposition~\ref{prop:Sklar:II:GEV}(ii)]
Recall that usc functions reach their suprema on compacta. Let $(K_n)_{n \in \nset}$ be a compact cover of $\domain$. For every $z \in \uscD$, we have $z(s) < \infty$ for all $s \in \domain$ if and only if $\sup_{s \in K_n} z(s) < \infty$ for all $n \in \nset$. The event $\bigcap_{n \in \nset} \{\sup_{s \in K_n} \xi^*(s) < \infty\}$ is thus the same as the
event $\{ \forall s \in \domain : \xi^*(s) < \infty \}$.
\end{rem}

\begin{rem}[Regarding the continuity assumption on $\theta$ in Proposition~\ref{prop:Sklar:II:GEV}]
According to Lemmas~\ref{lem:Qsp} and~\ref{lem:theta_cnt}, imposing the continuity of the GEV parameter vector $\theta(s)$ as a function of $s \in \domain$ is equivalent to imposing the upper semicontinuity of the function $s\mapsto F(x, \theta(s))$ for each fixed $x\in\rset$.
\end{rem}

\subsection{Examples}
\label{subsec:max-stable:examples}

In comparison to Proposition~\ref{prop:Sklar:II}, we have added to Proposition~\ref{prop:Sklar:II:GEV} the assumption that the margins be GEV. Although their distribution functions are continuous and strictly increasing on their support, this does not resolve the issues arising when the marginal distributions are not continuous in space, as in Example~\ref{ex:what_if_b_fails}. In Example~\ref{ex:nonContinuousTheta}, which parallels Example~\ref{ex:what_if_b_fails}, the GEV parameter function $s \mapsto \theta(s)$ is not continuous and pointwise standardization to unit Fr\'{e}chet margins produces a stochastic process whose trajectories are no longer usc almost surely.

\begin{ex}[What may happen without the continuity of $\theta$]
  \label{ex:nonContinuousTheta}
Consider two independent, unit-Fr\'{e}chet distributed variables $X$ and
$Y$. As in Example~\ref{ex:what_if_b_fails}, take $\domain = [0,2]$, and define 
\[
  \xi(s) 
  = 
  X \vee (Y \ind_{\{1\}}(s))
  =
  \begin{cases}
    X & \quad \text{ if $s \neq  1$,} \\
    X\vee Y& \quad \text{ if $s = 1$.}
  \end{cases}
\]
Then again, $\xi$ is a usc process. It is even a max-stable one with normalizing functions $a_n \equiv n$ and $b_n\equiv 0$. The marginal distribution functions are 
\begin{equation*}
  F(x; \theta(s)) = \Pr[ \xi(s) \le x ] =
  \begin{cases}
    e^{-1/x} & \quad \text{ if $s \neq 1$,} \\
    e^{-2/x}& \quad \text{ if $s = 1$,}
  \end{cases}
\end{equation*}
where $x \ge 0$. Then the function $s\mapsto F_s(x)$ is lower rather than upper semicontinuous, and the marginal GEV parameter vector is 
\begin{equation*}
\theta(s) = (\gamma(s), \mu(s), \sigma(s)) = 
  \begin{cases}
    (1, 1, 1) & \quad \text{ if $s \neq 1$,} \\
    (1, 2, 2) & \quad \text{ if $s = 1$,}
  \end{cases}
\end{equation*}
which is not continuous as a function of $s$. Standardizing to Fr\'{e}chet margins yields
\begin{equation*}
   \xi^*(s) = -1/\log F(\xi(s); \theta(s)) =
   \begin{cases}
  X 
 & \quad \text{ if $s \neq 1$,} \\
 (X\vee Y) / 2
 & \quad \text{ if $s = 1$.} \\
   \end{cases}
 \end{equation*}
The event $\{0 < X < Y < 2X\}$ has positive probability, and on this event, we have  $(X\vee Y )/2  = Y/2 < X$, so that the trajectory $s \mapsto \xi^*(s)$ is not usc.
\end{ex}

To conclude, we present a construction principle for simple max-stable usc processes. 
In combination with Proposition~\ref{prop:Sklar:I:max-stable}, this provides a device for the construction of max-stable usc processes with arbitrary GEV margins. The method is similar to the one proposed in Theorem~2 in \cite{schlather2002models}. Proving max-stability of the usc process that we construct requires special care, since max-stability in $\uscD$ does not follow from max-stability of the finite-dimensional distributions. 

\begin{ex}
Let $Y_1 > Y_2 > Y_3 > \ldots$ denote the points of a Poisson point process on $(0, \infty)$ with intensity measure $y^{-2} \, \mathrm{d} y$. Let $V, V_1, V_2, \ldots$ be iid usc processes, independent of the point process $(Y_i)_i$. Assume that $V$ satisfies the following properties:
\begin{compactitem}
\item
  $\PP[\inf_{s\in \domain}  V(s) \ge 0]  = 1$; 
\item
  $\PE[ \sup_{s\in \domain} V(s) ] < \infty$; 
\item 
  the mean function $f(s) = \PE[ V(s) ]$ is strictly positive and continuous on $\domain$.
\end{compactitem}
By Lemma~\ref{lem:inf_rv} below, $\inf_\domain V$ is indeed a random variable. Note that we do \emph{not} impose that $\inf_\domain V > 0$ almost surely.

Define a stochastic process $\xi$ on $\domain$ by
\begin{equation*}
  \xi(s) = \sup_{i \ge 1} Y_i \frac{1}{f(s)} V_i(s), \qquad s \in \domain.
\end{equation*}
We will show that $\xi$ is  `almost surely' a simple max-stable usc process, in the following sense:
\begin{compactenum}[(a)]
\item
  \label{uscTraj} 
  with probability one, the trajectories of $\xi$ are usc;
\item
  \label{measMap} 
for some set $\Omega_1\subset \Omega$ of probability one on which the trajectories $\xi(\point,\omega)$ are usc, the map $\xi : \Omega_1\to \uscD$ is hypo-measurable.
\end{compactenum}

Consider the space of nonnegative usc functions
\[
  \uscD_+ 
  = \{z \in \uscD : \forall s \in \domain, z(s)\ge 0 \} 
  = \bigcap_{s\in\domain} \{ z \in \uscD: \; z(s) \ge 0 \}.
\]
For each fixed $s\in\domain$, the set $\{z\in\uscD: z(s) \ge 0 \}$ is hypo-closed, since it is the complement of the hypo-open set $\{ z \in \uscD : \hypo z \cap \{ (s, 0 ) \} = \varnothing \}$. The set $\uscD_+$ is thus hypo-closed, as an intersection of hypo-closed sets.

Set $\Espace = \uscD_+\setminus\{0_\domain\}$,  where $0_\domain$ denotes for the null function on $\domain$. Since $\uscD$ is a  compact space with a countable basis \cite[Theorem~B.2, p.~399]{molchanov2005theory}, the space $\Espace$ is 
locally compact with a countable basis. Classical theory of point processes applies and it is possible to define Poisson processes on $\Espace $ by augmentation and/or continuous mappings.

Put $ W_i = V_i / f$. Then for $i\in\nset$, $W_i$ is a random element of $\uscD$ such that $W_i\in \Espace$ with probability one.  The point process $\Gamma$ defined by
\[
  \Gamma = \sum_{i\ge 1} \delta_{(Y_i,   W _i )}.
\]
is thus a Poisson process on $(0,\infty)\times \Espace $ with mean
measure $ \ud\Lambda(y,w) = y^{-2}\ud y \otimes \ud P_{W}(w)$, where
$P_{ W}$ is the law of $ W_1$ (\cite{resnick1987extreme},
Proposition~3.8).

The `product' mapping $T : (0,\infty) \times \Espace \to \Espace$ defined by $T(y, w) = yw$, for $y>0$ and $w\in \Espace$, is hypo-measurable. Provided that the image measure $\mu =\Lambda\circ T^{-1}$ is finite on compact sets of $\Espace$, we find that the point process
\[
  \Pi = \sum_{i\ge 1} \delta_{(Y_i W_i )}
\]
is a Poisson process with mean measure $\mu$ on $\Espace$
(\cite{resnick1987extreme}, Proposition~3.7). To check finiteness of $\mu$ on compact sets of $\Espace$, we must check that for $K\subset \domain$ and $x > 0$, writing
\[
  \mathcal F_{K\times \{x\}} 
  = \{z\in\uscD_+: \sup_{s\in K} z(s) \ge x \}, 
\]
we have $ \mu(\Fell_{K\times\{x\}})<\infty$. Indeed, the set $\mathcal F_{K\times \{x\}}$ is hypo-closed in $\uscD$, and since $\uscD$ is compact, the compact sets in $\Espace$ are the hypo-closed sets $F$ in $\uscD$  such that $F\subset \Espace$. Thus $\mathcal F_{K\times \{x\}}$ is compact. Also, any compact set in $\Espace$ must be contained in such a $\mathcal F_{K \times   \{x\}}$. 
Now, 
\[
  \begin{aligned}
    \mu(\Fell_{K\times\{x\}}) &= 
- \Lambda\big\{ (r,w) \in(0,\infty )\times \uscD:
 r \max_{s\in K}  w(s)  \ge x \big\}
\\ 
& = 
 \PE \Big[ \int_{0}^\infty   \ind{\Big\{ r \ge \inf_{s\in
     K} \frac{x}{W(s)}\Big\}} \frac{\ud r}{r^2} \Big] \\
& =  \frac{1}{x}
 \PE \Big[   \sup_{s\in
     K_j} W(s) \Big] \\
&\le \frac{1}{x} \frac{\PE[ \sup_\domain V ]}{\inf_{\domain} f},
  \end{aligned}
\]
which is finite by assumption on $V$.

To show \eqref{uscTraj},  we adapt an argument of \cite{gine1990max}(proof of Theorem~2.1). For fixed $K\subset \domain$ compact and $x>0$, we have $\mu(\Fell_{K\times\{x\}})<\infty$ and thus $\Pi(\Fell_{K\times\{x\}})<\infty$ almost surely. Thus, there exists a set $\Omega_1$ of probability one, such that the following two statements hold for all $\omega \in \Omega_1$:
\begin{enumerate}[(i)]
\item
  \label{positW}
  $W_i(\point,\omega)\in\Espace$ for all $i\in\nset$;
\item
  \label{finitePi}
  $\Pi(\Fell_{K\times\{x\}})(\omega)<\infty$ for every rational $x>0$ and every compact rectangle $K \subset \domain$ with rational vertices. 
\end{enumerate}

Take $\omega\in\Omega_1$. We need to show that for  $s\in\domain$ and $x\in\qset$ such that $\xi(s, \omega)<x$, we have $\limsup_{t\to s}\xi(t, \omega) \le x$. Fix $s$ and $x$ as above, which implies that $x>0$. Let $K_n\searrow \{s\}$ be a collection of compact rational rectangles as above such that $s$ is in the interior of $K_n$. Then $\Fell_{\{(s, x)\}} = \bigcap_{n\in\nset} \Fell_{K_n\times\{x\}}$ (the inclusion `$\subset$' is immediate; the inclusion `$\supset$' is obtained by choosing for $z\in \bigcap_{n\in\nset} \Fell_{K_n\times\{x\}}$ and for $n\in\nset$ a point $s_n\in K_n$ such that $z(s_n)\ge x$ and then observing that $s_n\to s$ and thus $z(s) \ge x$ by upper semicontinuity). From our choice of $\Omega_1$, we have $\Pi(\Fell_{K_1\times \{x\}}, \omega) <\infty$, so that the downward continuity property of the measure $\Pi(\point, \omega)$ applies and 
\[
  \Pi(\Fell_{\{(s,x)\}}, \omega)
  = \lim_{n \to \infty} \Pi(\Fell_{K_n\times \{x\}}, \omega).
\]
By our choice of $s$, $x$, and $\omega$, the left-hand side in the display is zero. Since the sequence on the right-hand side is integer valued, there exists $n_0$ such that for all $n\ge n_0$, we have $\Pi(\Fell_{K_n \times \{x\}}, \omega) = 0$. This implies that for $n \ge n_0$, we have
\[
  \sup_{t \in  K_n}  Y_i(\omega) W_i(t,\omega) < x, \qquad i\in\nset.
\]
Complete the proof of \eqref{uscTraj} by noting that
\[
\limsup_{t\to s} \xi(t,\omega) \le
\sup_{t\in K_{n_0}}\xi(t,\omega)
= \sup_{i\in\nset}\sup_{t\in K_{n_0}}
Y_i(\omega) W_i(t,\omega)
\le x.    
\]    

To show \eqref{measMap},  we need to show that, for any compact $K \subset \domain
\times \rset$, the set 
\[ 
  A = \{ \omega \in \Omega_1: \hypo \xi(\point,\omega) \cap K\neq\varnothing \}
\]
is a measurable subset of $\Omega$. Notice first that $\xi(\point,\omega)\in\uscD_+$ for $\omega\in\Omega_1$; use property~\eqref{positW} of $\Omega_1$. It follows that $A = \Omega_1$ as soon as $K$ is not a subset of $\domain \times (0, \infty)$. Assume $K \subset \domain\times (0, \infty)$. 

On $\Omega_1$, we have $\hypo \xi \cap K \neq \varnothing$ if and only if
$\Pi( \{ z \in \Espace : \hypo z \cap K \neq \varnothing \} ) \ge 1$. Now $\Fell_K = \{ z \in \Espace : \hypo z \cap K \neq
\varnothing \}$ is a hypo-measurable subset of $\Espace$, so that
$X = \Pi(\Fell_K)$ is a random variable.  
It follows that  $ A = X^{-1}([1,\infty])\cap \Omega_1$  is
measurable, yielding~\eqref{measMap}. 

A standard argument yields that the margins of $\xi$ are unit-Fr\'{e}chet. To show that $\xi$ is simple max-stable, we need to show that, for independent random copies $\xi_1, \ldots, \xi_n$ of $\xi$, the capacity functionals of $\bigvee_{i=1}^n \xi_i$ and $n \xi$ are the same. That is, we need to show that, for every compact set in $\domain \times \rset$, we have
\begin{equation}
\label{eq:max-stable:simple:capacity}
  \textstyle
  \Pr[ \hypo(\bigvee_{i=1}^n \xi) \cap K = \varnothing ] = \Pr[ \hypo(n \xi) \cap K = \varnothing ].
\end{equation}
The left-hand side is equal to
\begin{align}
\nonumber
  \textstyle
  \Pr[ (\bigcup_{i=1}^n \hypo \xi_i) \cap K = \varnothing ]
  &= \Pr[ \forall i = 1, \ldots, n : \hypo \xi_i \cap K = \varnothing ] \\
\label{eq:maximum:capacity}
  &= (\Pr[ \hypo \xi \cap K = \varnothing ])^n.
\end{align}
Without loss of generality, we may assume that
\[
  \textstyle
  K = \bigcup_{j=1}^p  (K_j \times \{x_j\}),
\] 
where $p\in\nset$ and where $K_j \subset \domain$ is compact and $x_j$
is real, for $j\in\{1,\ldots,p\}$. Indeed, the capacity functional of
a usc process is entirely determined by its values on such compacta
(\cite{molchanov2005theory}, p.~340). We may also choose $x_j>0$, since otherwise both sides of \eqref{eq:max-stable:simple:capacity} vanish. For such a set $K$, we have 
\begin{align*}
\lefteqn{
\PP[ \hypo \xi \cap K = \varnothing ]
} \\
& = \PP\Big[\Pi\big\{z\in\Espace:
\exists \, 1\le j\le p,  \max_{K_j} z \ge x_j \big\} = 0 \Big]    \\
& = \exp\Big( - \Lambda\big\{ (r,w) \in(0,\infty )\times \Espace: 
\exists  j\le p\,,\,  r \max_{s\in K_j}  w(s)  \ge x_j \big\}\Big)
\\ 
& = \exp\Big( - 
 \PE\Big[ \int_{0}^\infty   \ind{\Big\{ \exists j\,:\, r \ge \min_{s\in
     K_j} \frac{x_j}{W(s)}\Big\}} \frac{\ud r}{r^2} \Big]\Big) \\
& = \exp\Big( -
 \PE \Big[\int_{0}^\infty  \ind{\Big\{ r \ge \min_j\min_{s\in
     K_j} \frac{x_j}{W(s)}\Big\}} \frac{\ud r}{r^2}   \Big]\Big)
\\
& = \exp\Big( -
 \PE \Big[   \max_j \frac{ \max_{s\in
     K_j} W(s)}{x_j} \Big]\Big).
\end{align*}
For $\xi$ replaced by $n \xi$, we obtain the same result, but with $x_j$ replaced by $x_j / n$. In view of \eqref{eq:maximum:capacity}, the desired equality \eqref{eq:max-stable:simple:capacity} follows.

\end{ex}

\section{Conclusion}
\label{sec:conclusion}

The aim of the paper has been to extend Sklar's theorem from random vectors to usc processes. We have stated necessary and sufficient conditions to be able to construct a usc process with general margins by applying the pointwise quantile transformation to a usc process with standard uniform margins (Propositions~\ref{prop:Sklar:I} and~\ref{prop:Sklar:I:max-stable}). Furthermore, we have stated sufficient conditions for the pointwise probability integral transform to be possible for usc processes (Propositions~\ref{prop:Sklar:II}, \ref{prop:Sklar:II:GEV} and~\ref{prop:Sklar:II:max-stable}). These conditions imply in particular that the marginal distribution functions are continuous with respect to the space variable (Lemma~\ref{lem:equiv_usc_continuous_cdf}). We have also provided several examples of things that can go wrong when these conditions are not satisfied. However, finding \emph{necessary} and sufficient conditions remains an open problem.

The motivation has been to extend the margins-versus-dependence
paradigm used in multivariate extreme-value theory to max-stable usc
processes. The next step is to show that marginal standardization is
possible in max-domains of attraction too. One question, for instance,
is whether the standardized weak limit of the pointwise maxima of a
sequence of usc processes is equal to the weak limit of the pointwise
maxima of the sequence of standardized usc processes
(\cite{resnick1987extreme}, Proposition~5.10). Interesting
difficulties arise: weak convergence of finite-dimensional
distributions does not imply and is not implied by weak
hypoconvergence; Khinchin's convergence-of-types lemma does not apply
in its full generality to unions of random closed sets
(\cite{molchanov2005theory}, p.~254, `Affine normalization').  This topic will be the subject of further work.



\appendix

\section{Right-continuous quantile functions}
\label{app:quantile}

The \emph{right-continuous} quantile function, $Q$, of a random
variable $X$ taking values in $[-\infty, \infty]$ and with distribution
function $F(x) = \PP[X\le x]$, $x\in[-\infty,\infty]$, is defined as
\begin{equation}
\label{eq:quantile:right}
  Q(p) = \sup \{ x \in \rset : F(x) \le p \}, \qquad p \in [0, 1].
\end{equation}
By convention, $\sup \varnothing = -\infty$ and $\sup \rset =
\infty$.
The fact that $Q$ is right-continuous is stated in part (ii) of the
next lemma. The function $Q$ is denoted by $F_+^{-1}$ in
  \cite{pickands:1971}, p.~749, who mentions properties (ii) and (iv)
  of Lemma~\ref{prop:quantile:right}. The corresponding statements for
  the right-continuous inverse,
  $p \mapsto \inf \{ x \in \rset : F(x) \ge p \}$, are well-known, see
  for instance Section~0.2 in
  \cite{resnick1987extreme}. 

\begin{lem}
\label{prop:quantile:right}
Let $X$ be a random variable taking values in $[-\infty, \infty]$. Define $Q : [0, 1] \to [-\infty, \infty]$ as in \eqref{eq:quantile:right}. 
\begin{enumerate}[(i)]
\item
For all $p \in [0, 1]$, we have $Q(p) = \sup \{ x \in \rset : \PP(X <
x) \le p \}$.
\item
The function $Q$ is non-decreasing and right-continuous.
\item
For every $x \in [-\infty, \infty]$ and every $p \in [0, 1]$, we have $x \le Q(p)$ if and only if $\Pr[ X < x ] \le p$.
\item
If $V$ is uniformly distributed on $[0, 1]$, then the distribution function of $Q(V)$ is $F$, i.e., $Q(V)$ and $X$ are identically distributed.
\item
If the law of $X$ has no atoms in $[-\infty, \infty]$, then $\Pr[ X = Q(F(X)) ] = 1$.
\end{enumerate}
\end{lem}

\begin{proof}

\textit{(i)}
Fix $p \in [0, 1]$. Since $\Pr[X < x] \le \Pr[X \le x] = F(x)$, we have
\begin{align*}
  Q(p) 
  &:= 
  \sup \{ x \in \rset : F(x) \le p \} \\
  &\le \sup \{ x \in \rset : \Pr[ X < x ] \le p \} 
  := x_0.
\end{align*}
Conversely, if $Q(p)=+\infty$, there is nothing to prove. Assume then that $Q(p)<+\infty$, so that  $p<1$. Let $y > Q(p)$. We need to show that $y > x_0$ too, that is, $\Pr[X < y] > p$. But $\Pr[X < y] = \sup \{ F(z) : z < y \}$, and this supremum must be larger than $p$, since for all $z > Q(p)$ we have $F(z) > p$.

\textit{(ii)}
The sets $\{ x \in \rset : F(x) \le p \}$ becoming larger with $p$, the function $Q$ is non-decreasing. Next, we show that $Q$ is right continuous at any $p \in [0, 1]$. If
$Q(p) = \infty$, there is nothing to show, so suppose $Q(p) < \infty$
(in particular $p < 1$). Let $\eps > 0$. Then $Q(p) + \eps > Q(p)$ and
thus 
$F(Q(p) + \eps ) = p + \delta > p$
 for some $\delta > 0$. For $r < p+\delta$, we have 
$F( Q(p) + \eps )> r$
 too, and thus $Q(r) < Q(p) + \eps$, as required.

\textit{(iii)}
First suppose $x < Q(p)$; we show that $\Pr[X < x] \le p$. The case $x
= \infty$ is impossible, and if $x = -\infty$, then $\Pr[X < x] = 0
\le p$. So suppose that $x \in \rset$. Using statement
(i), there exists $y \in \rset$ with $x \le y \le Q(p)$ such that $\Pr[X < y] \le p$. But then also $\Pr[X < x] \le p$.

Second suppose that $x > Q(p)$; we show that $\Pr[X < x] >
p$. Clearly, we must have $Q(p) < \infty$, and so we can without loss
of generality assume that $x$ is real. But then $\Pr[X < x] > p$ by
statement (i).

Finally, consider $x = Q(p)$; we show that $\Pr[X < Q(p)] \le p$. If $Q(p) = -\infty$, then $\Pr[X < Q(p)] = 0 \le p$. If $Q(p) > -\infty$, then $\Pr[X < y] \le p$ for all $y < Q(p)$, and thus $\Pr[X < Q(p)] = \sup_{y < Q(p)} \Pr[X < y] \le p$ too. 

\textit{(iv)}
Without loss of generality, assume that $0 \le V \le 1$ (if not, then replace $V$ by $(V \vee 0) \wedge 1$, which is almost surely equal to $V$). Let $x \in [-\infty, \infty]$. By statement~(iii), we have $x \le Q(V)$ if and only if $\Pr[X < x] \le V$. As a consequence, $\Pr[ x \le Q(V) ] = \Pr[ \Pr[X < x] \le V ] = 1 - \Pr[X < x] = \Pr[x \le X]$. We conclude that $Q(V)$ and $X$ are identically distributed and thus that $Q(V)$ has distribution function $F$ too.

\textit{(v)}
By definition, $x \le Q(F(x))$ for every $x \in \rset$. For $x = -\infty$ and $x = +\infty$, the same inequality is trivially fulfilled too (recall that $F(\infty) = 1$ and $Q(1) = \infty$). As a consequence, $X \le Q(F(X))$. 

Conversely, let $\mathcal{P}$ be the collection of $p \in [0, 1]$ such that the set $\{ x \in \rset : F(x) = p \}$ has positive Lebesgue measure. These sets being disjoint for distinct $p$, the set $\mathcal{P}$ is at most countably infinite. If $x < Q(F(x))$, then there exists $y > x$ such that $F(x) = F(y)$ and thus $F(x) \in \mathcal{P}$. However, the law of $F(X)$ is standard uniform, so that $\Pr[ F(X) \in \mathcal{P} ] = 0$. Hence $X = Q(F(X))$ almost surely.
\end{proof}


\begin{lem}
\label{lem:limsupX:limsupQp}
Let $(X_n)_{n \in \mathbb{N}}$ be a sequence of random variables defined on the same probability space and taking values in $[-\infty, \infty]$. Let $Q_n$ and $Q$ be the right-continuous quantile functions \eqref{eq:quantile:right} of $X_n$ and $\limsup_{n \to \infty} X_n$, respectively. Then
\[
  Q(p) \ge \limsup_{n \to \infty} Q_n(p),
  \qquad p \in [0, 1].
\]
\end{lem}

\begin{proof}
If $Q(p) = \infty$, there is nothing to show, so suppose $Q(p) < \infty$. Let $y > Q(p)$; we will show that $y \ge \limsup_{n \to \infty} Q_n(p)$. This being true for all $y > Q(p)$, we will have proved the proposition.

By Lemma~\ref{prop:quantile:right}(i), we have $\Pr[ \limsup_{n \to \infty} X_n < y
] > p$. By Fatou's lemma,
there exists a positive integer $n(y)$ such that $\Pr[ X_n < y ] > p$ for all integer $n \ge n(y)$. But for such $n$, we have $y > Q_n(p)$ too. Hence, $y \ge \limsup_{n \to \infty} Q_n(p)$, as required.
\end{proof}

\section{Additional examples related to Proposition~\ref{prop:Sklar:II}}
\label{ap:examplesSklar2}

\begin{ex}[Condition~(b) in Proposition~\ref{prop:Sklar:II} is not necessary]
\label{ex:b_not_necess}
Let $\domain = [-1,1]$ and let $X$ and $V$ be independent random variables, $X$ standard normal and $V$ uniform on $[0, 1]$. Define
\begin{equation*}
  \xi(s) = 
  \begin{cases}
    X &  \text{if $-1 \le s \le 0$,} \\
    X-1 &  \text{if $0 < s < V$,} \\
    X & \text{if $V \le s \le 1$.}
  \end{cases}
\end{equation*}
Let $\Phi$ be the standard normal cumulative distribution function and choose $x \in \rset$. Then $F_s(x) = \Pr[ \xi(s) \le x ] = \Phi(x)$ if $s \in [-1, 0]$, while for $s \in (0, 1]$, we have
\begin{align*}
  F_s(x) 
  &= \Pr[s < V] \, \Pr[X-1 \le x] + \Pr[V \le s] \, \Pr[X \le x] \\
  &= (1-s) \, \Phi(x+1) + s \, \Phi(x).
\end{align*}
The function $s \mapsto F_s(x)$ is constant on $s \in [-1, 0]$ while it decreases linearly from $\Phi(x+1)$ to $\Phi(x)$ for $s$ from $0$ to $1$, the right-hand side limit at $0$ being equal to $\Phi(x+1)$, which is greater than $\Phi(x)$, the value at $s = 0$ itself. Hence the function $s \mapsto F_s(x)$ is lower but not upper semicontinuous, and condition~(b) in Proposition~\ref{prop:Sklar:II} does not hold. Nevertheless, the random variables $Z(s) = F_s(\xi(s))$, $s \in \domain$, are given as follows:
\[
  Z(s) = F_s(\xi(s)) =
  \begin{cases}
    \Phi(X) & \text{if $-1 \le s \le 0$,} \\
    (1-s) \, \Phi(X) + s \, \Phi(X-1) & \text{if $0 < s < V$,} \\
    (1-s) \, \Phi(X+1) + s \, \Phi(X) & \text{if $V \le s \le 1$.}
  \end{cases}
\]
The trajectory of $Z$ is continuous at $s \in [-1, 1] \setminus \{
V \}$ and usc at $s = V$, hence usc overall.
\end{ex}

\begin{ex}[$\xi$ and $\tilde{\xi}$ in Proposition~\ref{prop:Sklar:II} may be different in law (1)]
\label{ex:not_equal_in_law}
Let $X$ and $Y$ be independent, uniform random variables on $[0, 1]$. For $s \in \domain = [0, 1]$, define
\[
  \xi(s) = X + \ind(Y = s) =
  \left\{
  \begin{array}{ll}
    X & \text{if $s \neq Y$,} \\
    X + 1 & \text{if $s = Y$.}
  \end{array}
  \right.
\]
Since $\Pr[Y = s] = 0$ for every $s \in [0, 1]$, the law of $\xi(s)$ is standard uniform. Conditions~(a) and~(b) in Proposition~\ref{prop:Sklar:II} are trivially fulfilled with $F_s(x) = (x \vee 0) \wedge 1$ for $x \in [-\infty, \infty]$ and $Q_s(p) = p$ for $p \in [0, 1)$ while $Q_s(1) = \infty$. We obtain
\[
  Z(s) = F_s(\xi(s)) =
  \left\{
  \begin{array}{ll}
    X & \text{if $s \neq Y$,} \\
    1 & \text{if $s = Y$,}
  \end{array}
  \right.
\]
and thus
\[
  \tilde{\xi}(s) = Q_s(Z(s)) =
  \left\{
  \begin{array}{ll}
    X & \text{if $s \neq Y$,} \\
    \infty & \text{if $s = Y$.}
  \end{array}
  \right.
\]
The processes $\tilde{\xi}$ and $\xi$ have different capacity functionals and thus a different distribution as random elements in $\uscD$: for $K = [0, 1] \times \{x\}$ with $x > 2$, we have $\Pr[\hypo \xi \cap K \neq \varnothing] = 0$ while $\Pr[ \hypo \tilde{\xi} \cap K \neq \varnothing ] = 1$.
\end{ex}

\begin{ex}[$\xi$ and $\tilde{\xi}$ in Proposition~\ref{prop:Sklar:II} may be different in law (2)]
\label{ex:not_equal_in_law2}
Let $X$ and $Y$ be independent random variables, with $X$ uniformly distributed on $[0, 1] \cup [2, 3]$ and $Y$ uniformly distributed on $[0, 1]$. For $s \in \domain = [0, 1]$, define
\[
  \xi(s) = X \vee (1.5 \times \ind(Y = s))
  =
  \begin{cases}
    X & \text{if $s \neq Y$ or $X>1$,} \\
    1.5 & \text{if $s = Y$ and $X \le 1$.}
  \end{cases}
\]
Then $\Pr[\xi(s) = X] = 1$ for all $s \in [0, 1]$, so that the marginal distribution and quantile functions $F_s$ and $Q_s$ do not depend on $s$ and are equal to those of $X$, denoted by $F_X$ and $Q_X$. The random variable $U = F_X(X)$ is uniformly distributed on $[0, 1]$. Since $F_X(1.5) = F_X(1) = 0.5$, we have
\[
  Z(s) = F_X(\xi(s))
  =
  \begin{cases}
    U & \text{if $s \neq Y$ or $X > 1$,} \\
    0.5 & \text{if $s = Y$ and $X \le 1$.}
  \end{cases}
\]
However, as $Q_X(0.5) = 2$, we obtain
\[
  \tilde{\xi}(s) = Q_X(Z(s))
  =
  \begin{cases}
    X & \text{if $s \neq Y$ or $X > 1$,} \\
    2 & \text{if $s = Y$ and $X \le 1$.}
  \end{cases}  
\]
On the event $\{X \le 1\}$, which occurs with probability one half,
the hypographs of $\xi$ and $\tilde \xi$ are different.
\end{ex}

\section{Lower bounds of usc processes}

\begin{lem}
\label{lem:lower_bound}
If $\xi$ is a usc process, then the function $\ell(s) =
\sup \{ x \in \rset : \Pr[ \xi(s) < x ] = 0 \}$, for $s \in \domain$,
is usc. Moreover, $\xi \vee \ell$ is a usc process too, and we have $\Pr[ \xi \ge \ell ] = \Pr[ \xi = \xi \vee \ell ] = 1$.
\end{lem}

\begin{proof}
The function $\ell$ is equal to the function $s \mapsto Q_s(0)$, with $Q_s$ the right-continuous quantile function \eqref{eq:quantile:right} of $\xi_s$. By Lemma~\ref{lem:Qsp}, the function $\ell$ is usc, and by Example~\ref{ex:transfo:simple}, the map $\uscD \to \uscD : z \mapsto z \vee \ell$ is hypo-measurable, so that $\xi \vee \ell$ is a usc process too. 

By Lemma~\ref{lem:z_separable}, there exists a countable subset, $\mathbb{Q}_\ell$, of $\domain$ with the following property: for all $x \in \uscD$, we have $x(s) \ge \ell(s)$ for all $s \in \domain$ if and only if $x(t) \ge \ell(t)$ for all $t \in \mathbb{Q}_\ell$. Since $\ell(t) = Q_t(0)$, we have $\Pr[ \xi(t) \ge \ell(t) ] = 1$ for all $t \in \mathbb{Q}_\ell$; see Lemma~\ref{prop:quantile:right}(iii). Since $\mathbb{Q}_\ell$ is countable, also $\Pr[ \forall t \in \mathbb{Q}_\ell : \xi(t) \ge \ell(t) ] = 1$. By the property of $\mathbb{Q}_\ell$ mentioned earlier, the event $\{ \forall t \in \mathbb{Q}_\ell : \xi(t) \ge \ell(t) \}$ is equal to both $\{ \xi \ge \ell \}$ and $\{ \xi = \xi \vee \ell\}$.
\end{proof}

\begin{lem}
\label{lem:z_separable}
For every $z \in \uscD$, there exists a countable set $\mathbb{Q}_z \subset \domain$ such that
\begin{equation}
\label{eq:zQ}
  z(s) = \inf_{\eps > 0} \sup_{t \in \mathbb{Q}_z : d(s, t) \le \eps} z(t), \qquad s \in \domain.
\end{equation}
In particular, for every $x \in \uscD$, we have $x(s) \ge z(s)$ for all $s \in \domain$ if and only if $x(t) \ge z(t)$ for all $t \in \mathbb{Q}_z$.
\end{lem}

\begin{proof}
The set $\hypo z$ is a subset of the metrizable separable space $\domain \times \rset$. Hence, it is separable too. Let $\mathbb{Q}$ be a countable, dense subset of $\hypo z$. Let $\mathbb{Q}_z$ be the set of $t \in \domain$ such that $(t, x) \in \mathbb{Q}$ for some $x \in \rset$. Then $\mathbb{Q}_z$ is a countable subset of $\domain$.

Let $y(s)$ denote the right-hand side of \eqref{eq:zQ}. Since $z$ is usc, we have $z(s) \ge y(s)$ for all $s \in \domain$.

Conversely, let $s \in \domain$. If $z(s) = -\infty$, then trivially $y(s) \ge z(s)$. Suppose $z(s) > -\infty$. Let $-\infty < \alpha < z(s)$, so that $(s, \alpha) \in \hypo z$. Find a sequence $(t_n, \alpha_n) \in \mathbb{Q}$ such that $(t_n, \alpha_n) \to (s, \alpha)$ as $n \to \infty$. Then $(t_n, \alpha_n) \in \hypo z$ and thus $z(t_n) \ge \alpha_n$ for all $n$. Moreover, $t_n \to s$ and $\alpha_n \to \alpha$ as $n \to \infty$. It follows that $y(s) \ge \alpha$. Since this is true for all $\alpha < z(s)$, we find $y(s) \ge z(s)$.

We prove the last statement. Let $x \in \uscD$ and suppose that $x(t) \ge z(t)$ for all $t \in \mathbb{Q}_z$. The function $x$ is equal to its own usc hull, i.e., 
\[ 
  x(s) = \inf_{\eps > 0} \sup_{t \in \domain : d(s, t) \le \eps} x(t), \qquad s \in \domain. 
\]
Combine this formula together with \eqref{eq:zQ} to see that $x(s) \ge z(s)$ for all $s \in \domain$.
\end{proof}

\begin{lem}
\label{lem:inf_rv}
If $\xi$ is a usc process, then $\inf_{s \in F} \xi(s)$ is a random variable for any closed set $F \subset \domain$.
\end{lem}

\begin{proof}
Let $\mathbb{Q}$ be a countable, dense subset of $F$. Since every $\xi(s)$ is a random variable, it suffices to show that $\inf_{s \in F} \xi(s) = \inf_{s \in \mathbb{Q}} \xi(s)$. The inequality `$\le$' is trivial. To see the other inequality, suppose that $x$ is such that $\xi(s) \ge x$ for all $s \in \mathbb{Q}$. Then $(s, x) \in \hypo \xi$ for all $s \in \mathbb{Q}$, and thus $(s, x) \in \hypo \xi$ for all $s \in F$, since $\hypo \xi$ is closed. It follows that $\xi(s) \ge x$ for all $s \in F$.
\end{proof}

\section{Continuity of the GEV parameter}


Recall the GEV distributions from Subsection~\ref{subsec:gev}.

\begin{lem}
\label{lem:continuityGEVparameter}
Let $ F_n = F(\point, \theta_n)$, $n \ge 0$, be GEV distribution functions with associated GEV parameters $  \theta_n = (\mu_n,\sigma_n,\gamma_n)$. If $(F_n)_{n}$ converges weakly to $F_0$, then also $\lim_{n \to \infty} \theta_n = \theta_0$ in $\rset\times\rset\times(0,\infty)$. 
\end{lem}

\begin{proof}
By continuity of the GEV distribution, weak convergence is the same as pointwise convergence, and thus $\lim_{n \to \infty} F_n(x) = F_0(x)$ for all $x\in\rset$.

Recall the expression \eqref{eq:gev:Q} of the quantile function $Q(\point;\theta)$. Pointwise convergence of monotone functions implies pointwise convergence of their inverses at continuity points of the limit \cite[Chapter~0]{resnick1987extreme}. Setting $p = e^{-1}$, we obtain 
\[
  \mu_n = Q(e^{-1}; \theta_n) \to Q(e^{-1}; \theta_0) = \mu_0, \qquad n \to \infty.
\]
As a consequence, for $p \in (0,1)$,
\begin{align}
 \nonumber
  Q(p; \gamma_n, 0, \sigma_n) 
  &= Q(p; \theta_n) - \mu_n \\
  &\to 
  Q(p; \theta_0) - \mu_0
  = Q(p; \gamma_0, 0, \sigma_0),
  \qquad n \to \infty.
\label{eq:cvQuantile-mu}
\end{align}
This implies that for $x, y > 0$ such that $y \neq 1$, 
\begin{equation*}
  \lim_{n \to \infty} \frac{x^{\gamma_n} -1}{y^{\gamma_n }-1} 
  = \frac{x^{\gamma_0} -1}{y^{\gamma_0}-1}.  
\end{equation*}
For $\gamma_n = 0$, the above expressions are to be understood as $\log(x)/\log(y)$. A subsequence argument then yields that $(\gamma_n)_n$ must be bounded, and a second subsequence argument confirms that $\gamma_n \to \gamma_0$ as $n \to \infty$. Combine this convergence relation with \eqref{eq:cvQuantile-mu} and use the identity $Q_n(p; \gamma_n, 0, \sigma_n) = \sigma_n \, Q(p; \gamma_n, 0, 1)$ to conclude that $\sigma_n \to \sigma_0$ as $n \to \infty$.
\end{proof}

\section*{Acknowledgments}

Part of A. Sabourin's work  has been  funded by the `AGREED' project from the PEPS JCJC program (INS2I, CNRS) and by the chair `Machine Learning for Big Data' from Télécom ParisTech. J. Segers gratefully acknowledges funding by contract ``Projet d'Act\-ions de Re\-cher\-che Concert\'ees'' No.\ 12/17-045 of the ``Communaut\'e fran\c{c}aise de Belgique'' and by IAP research network Grant P7/06 of the Belgian government (Belgian Science Policy).


%
\bibliographystyle{apalike} 
\bibliography{usc3}

\end{document}